\documentclass[12pt]{amsart}

\usepackage{amsmath}
\usepackage{amsthm}
\usepackage{amssymb}
\usepackage{amsfonts}
\usepackage{amsrefs}
\usepackage{color}
\usepackage{tikz}

\tikzstyle{shaded}=[fill=red!10!blue!20!gray!30!white]
\tikzstyle{shaded line}=[double=red!10!blue!20!gray!30!white, double distance=1.5mm, draw=black]
\tikzstyle{unshaded}=[fill=white]
\tikzstyle{unshaded line}=[double=white, double distance=1.5mm, draw=black]
\tikzstyle{Tbox}=[circle, draw, thick, fill=white, opaque,]
\tikzstyle{empty box}=[circle, draw, thick, fill=white, opaque, inner sep=2mm]
\tikzstyle{background rectangle}= [fill=red!10!blue!20!gray!40!white,rounded corners=2mm] 
\tikzstyle{on}=[very thick, red!50!blue!50!black]
\tikzstyle{off}=[gray]

\tikzstyle{traces}=[scale=.2, inner sep=1mm]
\tikzstyle{quadratic}=[scale=.4, inner sep=1mm, baseline]
\tikzstyle{annular}=[scale=.7, inner sep=1mm, baseline]
\tikzstyle{make triple edge size}= [scale=.4, inner sep=1mm,baseline] 
\tikzstyle{icosahedron network}=[scale=.3, inner sep=1mm, baseline]
\tikzstyle{ATLsix}=[scale=.25, baseline]
\tikzstyle{TL12}=[scale=.15,baseline]
\tikzstyle{PAdefn}=[scale=.7,baseline]
\tikzstyle{TLEG}=[scale=.5,baseline]

\makeatletter

\newtheorem{lemma}{Lemma}[section]
\newtheorem{definition}[lemma]{Definition}
\newtheorem{theorem}[lemma]{Theorem}
\newtheorem{proposition}[lemma]{Proposition}
\newtheorem{remark}[lemma]{Remark}
\newtheorem{corollary}[lemma]{Corollary}
\newtheorem{conjecture}[lemma]{Conjecture}
\newtheorem{example}[lemma]{Example}

\newtheorem{application}[lemma]{Application}

\newenvironment{claim}[1]{\par\noindent\underline{Claim:}\space#1}{}
\newenvironment{claimproof}[1]{\par\noindent\underline{Proof:}\space#1}{\hfill $\blacksquare$}

\makeatother

 \title[Dual Ore's theorem for distributive interval, small index]{Dual Ore's theorem for distributive intervals of small index}
  \author[Sebastien Palcoux]{Sebastien Palcoux}
\address{Institute of Mathematical Sciences, Chennai, India}
\email{palcoux@imsc.res.in}

\begin{document}

\begin{abstract}
This paper proves a dual version of a theorem of Oystein Ore for every distributive interval of finite groups $[H,G]$ of index $|G:H|<9720$, and for every boolean interval of rank $<7$. It has applications to representation theory for every finite group.
\end{abstract}

\maketitle

\section{Introduction} 
Oystein Ore has proved that a finite group is cyclic if and only if its subgroup lattice is distributive  \cite{or}. He has extended one side as follows:   

\begin{theorem}[\cite{or}] \label{oreintro}
 Let $[H,G]$ be a distributive interval of finite groups. Then $\exists g  \in G$ such that $\langle Hg \rangle = G$.
\end{theorem}

We have conjectured the following dual version of this theorem:  

\begin{conjecture} \label{dualoreintro}
 Let $[H,G]$ be a distributive interval of finite groups. Then $\exists V$ irreducible complex representation of $G$, with $G_{(V^H)} = H$ (Definition \ref{fixstab}); this property will be called \textbf{linearly primitive}.
\end{conjecture} 
The interval $[1,G]$ is linearly primitive if and only if $G$ is linearly primitive (i.e. admits a faithful irreducible complex representation). 
We will see that Conjecture \ref{dualoreintro} reduces to the boolean case, because a distributive interval is \textit{bottom boolean} (i.e. the interval generated by its atoms is boolean).
As application, Conjecture \ref{dualoreintro} leads to a new bridge between combinatorics and representation theory of finite groups:

\begin{definition}
Let $[H,G]$ be any interval. We define the combinatorial invariant $bb\ell(H,G)$ as the minimal length $\ell$ for a chain of subgroups $$H=H_0 < H_1 < \dots < H_{\ell} = G$$ with $[H_i,H_{i+1}]$ bottom boolean. Then, let $bb\ell(G):=bb\ell(1,G)$.
\end{definition}

\begin{application} Assuming Conjecture \ref{dualoreintro}, $bb\ell(G)$ is a non-trivial upper bound for the minimal number of irreducible complex representations of $G$ generating (for $\oplus$ and $\otimes$) the left regular representation.
\end{application} 

\begin{remark} If the normal subgroups of $G$ are also known, note that $$cf\ell(G):=min \{bb\ell(H,G) \mid H \text{ core-free}\}$$ is a better upper bound.  For more details on the applications, see \cite{pa,bp}. \end{remark}

This paper is dedicated to prove Conjecture \ref{dualoreintro} for $[H,G]$ boolean of rank $<7$, or distributive of index $|G:H| < 9720$. For so, we will use the following new result together with two former results:  

\begin{theorem} \label{theo0}
Let $[H,G]$ be a boolean interval and $L$ a coatom with $|G:L| = 2$. If $[H,L]$ is  linearly primitive, then so is $[H,G]$.
\end{theorem} 

\begin{theorem}[\cite{pa}] \label{theo1}
A distributive interval $[H,G]$ with $$\sum_{i=1}^{n} \frac{1}{\vert K_i : H  \vert} \le 2$$   for $K_1, \dots , K_n$  the minimal overgroups of $H$, is  linearly primitive. 
\end{theorem}

\begin{theorem}[\cite{bp}] \label{theo2}
A boolean interval $[H,G]$ with a (below) nonzero dual Euler totient,  is  linearly primitive. $$\hat{\varphi}(H,G):=\sum_{K \in [H,G]} (-1)^{\ell(H,K)} |G:K|$$ 
\end{theorem}

\begin{remark}[\cite{bp}] The Euler totient  $\varphi(H,G)=\sum (-1)^{\ell(K,G)} |K:H|$ is the number of cosets $Hg$ with $\langle Hg \rangle = G$, so $\varphi>0$ by Theorem \ref{oreintro}; but in general $\hat{\varphi} \neq \varphi$. We extend $\varphi$ to any distributive interval as $$\varphi(H,G) = |T:H| \cdot \varphi(T,G)$$ with $[T,G]$ the top interval of $[H,G]$, so that for $n = \prod_i p_i^{n_i}$,  $$\varphi(1,\mathbb{Z}/n) = \prod_ip_i^{n_i-1} \cdot  \prod_i(p_i-1)$$ which is the usual Euler totient $\varphi (n)$. Idem for $\hat{\varphi}$ and bottom interval.
\end{remark}

We will also translate our planar algebraic proof of Theorem \ref{theo1} in the group theoretic framework (one claim excepted).

\tableofcontents

\section{Preliminaries on lattice theory} \label{2}
\begin{definition}
A lattice $(L, \wedge , \vee)$ is a partially ordered set (or poset) $L$  in which every two elements $a,b$ have a unique supremum (or join) $a \vee b$ and a unique infimum (or meet) $a \wedge b$.
\end{definition}  

\begin{example}  Let $G$ be a finite group. The set of subgroups $ K \subseteq G$ is a lattice,  denoted by $\mathcal{L}(G)$, ordered by $\subseteq$, with $K_1 \vee K_2 = \langle K_1,K_2 \rangle$ and $K_1 \wedge K_2 =  K_1 \cap K_2 $. \end{example}
 
\begin{definition}
A sublattice of $(L, \wedge , \vee)$ is a subset $L' \subseteq L$ such that $(L', \wedge , \vee)$ is also a lattice. Let $a,b \in L$ with $a \le b$, then the interval $[a,b]$ is the sublattice $\{c \in L \ \vert \ a \le c \le b \}$. \end{definition}  

\begin{definition}
A finite lattice $L$ admits a minimum and a maximum, called $\hat{0}$ and $\hat{1}$. 
\end{definition} 

\begin{definition} \label{defatom}
An atom is an element $a \in L$ such that $$\forall b \in L, \ \hat{0} < b \le a \Rightarrow a=b.$$ A coatom is an element $c \in L$ such that $$\forall b \in L, \  c \le b < \hat{1} \Rightarrow b=c.$$
\end{definition} 

\begin{definition} \label{top}  \label{bottom}
The top interval of a finite lattice $L$ is the interval $[t,\hat{1}]$ with $t$ the meet of all the coatoms. The bottom interval is  the interval $[\hat{0},b]$ with $b$ the join of all the atoms.
\end{definition}

\begin{definition}
The length of a finite lattice $L$ is the greatest length $\ell$ of a chain $ 0 < a_1 < a_2 < \cdots < a_{\ell} = 1$ with $a_i \in L$.
\end{definition} 

\begin{definition}
A lattice $(L, \wedge , \vee)$ is distributive if $\forall a,b,c \in L$: 
$$a \vee (b \wedge c) = (a \vee b) \wedge (a \vee c)$$  
 (or equivalently,  $\forall a,b,c \in L, \  a \wedge (b \vee c) = (a \wedge b) \vee (a \wedge c)$). \end{definition}  
 
 \begin{lemma} \label{distri}
The reverse lattice and the sublattices of a  distributive lattice are also distributive. Idem for concatenation and direct product. \end{lemma}    

\begin{definition}\label{comp}
A distributive lattice is called boolean if any element $b$ admits a unique complement $b^{\complement}$ (i.e. $b \wedge b^{\complement} = \hat{0}$ and $b \vee b^{\complement} = \hat{1}$). 
\end{definition}

\begin{example}
The subset lattice of $\{1,2, \dots, n \}$, for union and intersection, is called the boolean lattice $B_n$ of rank $n$ (see $B_3$ below).
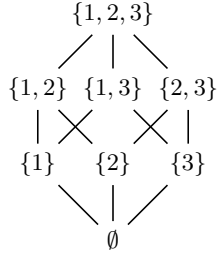
\begin{scriptsize}
\begin{center} $\begin{tikzpicture}
\node (A1) at (0,0) {$\{1,2,3\}$};
\node (A2) at (-1,-1) {$\{1,2\}$};
\node (A3) at (0,-1) {$\{1,3\}$};
\node (A4) at (1,-1) {$\{2,3\}$};
\node (A8) at (-1,-2) {$\{1\}$};
\node (A9) at (0,-2) {$\{2\}$};
\node (A10) at (1,-2) {$\{3\}$};
\node (A14) at (0,-3) {$\emptyset$};
\tikzstyle{segm}=[-,>=latex, semithick]
\draw [segm] (A1)to(A2); \draw [segm] (A3)to(A8); 
\draw [segm] (A1)to(A3); \draw [segm] (A3)to(A10);  
\draw [segm] (A1)to(A4); \draw [segm] (A4)to(A9);  
\draw [segm] (A14)to(A9); \draw [segm] (A14)to(A10); 
 \draw [segm] (A14)to(A8); \draw [segm] (A2)to(A9);
\draw [segm] (A2)to(A8); \draw [segm] (A4)to(A10); 

\end{tikzpicture} 
$ \end{center}
\end{scriptsize}
 \end{example}
 
\begin{remark}
Any finite boolean lattice is isomorphic to some $B_n$.  
\end{remark}

\begin{theorem}[Birkhoff's representation theorem or FTFDL \cite{sta}]
Any finite distributive lattice embeds into a finite boolean lattice.  
\end{theorem}

%

\begin{corollary} \label{topBn} \label{bottomBn} 
The top and bottom intervals of a distributive lattice are boolean.
 \end{corollary}   
 \begin{proof}
See \cite[items a-i p254-255]{sta}, together with Lemma \ref{distri}. \end{proof} 

%

\section{A dual version of Ore's theorem}

In this section, we will state the dual version of Ore's theorem, and prove it for any boolean interval of rank $\le 4$, after Theorem \ref{theo1} proof.


\begin{definition} \label{fixstab} Let $W$ be a representation of a group $G$, $K$ a subgroup of $G$, and $X$ a subspace of $W$. We define the \textit{fixed-point subspace} $$W^{K}:=\{w \in W \ \vert \  kw=w \ , \forall k \in K  \}$$ and the \textit{pointwise stabilizer subgroup} $$G_{(X)}:=\{ g \in G \  \vert \ gx=x \ , \forall x \in X \}$$  \end{definition} 

\begin{lemma}{\cite[Section 3.2]{bp}} \label{tech} Let $G$ be a finite group, $H,K$ two subgroups, $V$ a complex representation of $G$ and $X,Y$ two subspaces. Then
\begin{itemize}
\item[(1)] $H \subseteq K \Rightarrow V^{K} \subseteq V^{H}$
\item[(2)] $X \subseteq Y \Rightarrow G_{(Y)} \subseteq G_{(X)}$
\item[(3)] $V^{H \vee K} = V^H \cap V^K $
\item[(4)] $H \subseteq G_{(V^H)}$
\item[(5)] $ V^{G_{(V^H)}} = V^H$
\item[(6)] $[H \subseteq K$ and $V^{K} \subsetneq V^{H}]$ $\Rightarrow K \not \subseteq G_{(V^H)}$
\end{itemize}

\end{lemma}
\begin{lemma}{\cite{bp}} \label{indexrep}
Let $V_1, \dots , V_r$ be  the irreducible complex representations of a finite group $G$ (up to equivalence), and $H$ a subgroup. Then $$|G:H| = \sum_{i=1}^r \dim(V_i)\dim(V_i^H).$$   
\end{lemma} 

 \begin{definition} 
An interval of finite groups $[H,G]$ is called linearly primitive if there is an irreducible complex representation $V$ of $G$ such that $G_{(V^H)} = H$.
\end{definition}

\begin{remark} The interval $[1,G]$ is linearly primitive iff $G$ is linearly primitive (i.e. it admits an irreducible faithful complex representation). 
\end{remark}
The dual version of Ore's Theorem \ref{oreintro} is the following:
\begin{conjecture} A distributive interval $[H,G]$ is linearly primitive.
\end{conjecture}

 \begin{lemma} \label{B1}
A boolean interval $[H,G]$ of rank $1$ is linearly primitive. \end{lemma}    
\begin{proof} Note that $[H,G]$  is of rank $1$ iff $H$ is a maximal subgroup of $G$. Let $V$ be a non-trivial irreducible complex representation of $G$ with $V^H \neq \emptyset$, by Lemma \ref{tech} (4), $H \subseteq G_{(V^H)}$. If $G_{(V^H)} = G$ then $V$ must be trivial (by irreducibility), so by maximality $G_{(V^H)} = H$.
\end{proof}  

\begin{lemma}{\cite[Lemma 3.37]{bp}} \label{bottomprim}
An interval $[H,G]$ is linearly primitive if its bottom interval $[H,B]$ is so (see Definition \ref{bottom}).
\end{lemma}

 \begin{proposition} \label{<=1}
An interval $[H,G]$ satisfying $$\sum_{i=1}^{n} \frac{1}{\vert K_i : H  \vert} \le 1$$  with $K_1, \dots , K_n$  the minimal overgroups of $H$, is linearly primitive. \end{proposition}    
\begin{proof} First, by Lemmas \ref{B1}, \ref{bottomprim}, we can assume $n>1$. By assumption $\sum_{i=1}^{n} \frac{|G:H|}{\vert K_i : H  \vert} \le |G:H|$, so $\sum_{i=1}^{n} |G:K_i| \le |G:H|$. Let $V_1, \dots , V_r$ be  the irreducible complex representations of $G$. By Lemma \ref{indexrep}  $$\sum_{i=1}^{n} |G:K_i| = \sum_{i=1}^{n} \sum_{\alpha = 1}^r \dim(V_{\alpha})\dim(V_{\alpha}^{K_i}) = \sum_{\alpha = 1}^r \dim(V_{\alpha})[\sum_{i=1}^{n} \dim(V_{\alpha}^{K_i})].$$    

If $\forall \alpha$, $\sum_i V_{\alpha}^{K_i} = V_{\alpha}^H$, then $$\sum_{i=1}^{n} \dim(V_{\alpha}^{K_i}) \ge \dim(V_{\alpha}^{H}),$$ and so $\sum_{i=1}^{n} |G:K_i| \ge |G:H|$, but $\sum_{i=1}^{n} |G:K_i| \le |G:H|$, then $\sum_{i=1}^{n} |G:K_i| = |G:H|$. So $\forall \alpha$, $$\sum_{i=1}^{n} \dim(V_{\alpha}^{K_i}) = \dim(V_{\alpha}^{H}),$$ but for $V_1$ trivial, we get that $n=\sum_{i=1}^{n} \dim(V_{1}^{K_i})= \dim(V_{1}^{H}) = 1$, contradiction with $n>1$.    

Else there is $\alpha$ such that $\sum_i V_{\alpha}^{K_i} \subsetneq V_{\alpha}^H$, then by Lemma \ref{tech} (6),   $K_i \not \subseteq G_{(V_{\alpha}^H)} \ \forall i$, which means that $G_{(V_{\alpha}^H)}=H$  by minimality. \end{proof}

\begin{corollary} \label{B2}
If a subgroup $H$ of $G$ admits at most two minimal overgroups, then $[H,G]$ is linearly primitive. In particular, a boolean interval of rank $n \le 2$ is linearly primitive.
\end{corollary}
\begin{proof}
$\sum_{i} \frac{1}{\vert K_i : H  \vert} \le \frac{1}{2} + \frac{1}{2} = 1$; the result follows by Proposition \ref{<=1}.
\end{proof}  

%

We can upgrade Proposition \ref{<=1} in the distributive case as follows:

 \begin{theorem} \label{<=2}
A distributive interval $[H,G]$ satisfying $$\sum_{i=1}^{n} \frac{1}{\vert K_i : H  \vert} \le 2$$  with $K_1, \dots , K_n$  the minimal overgroups of $H$, is linearly primitive. \end{theorem}   
\begin{proof}  
By Lemma \ref{bottomprim}, Corollaries \ref{bottomBn} and \ref{B2}, we can assume the interval to be boolean of rank $n>2$.  

If $\exists \alpha$ such that 
\begin{equation} \label{Xalpha}
 \sum_{i,j,i \neq j}V_{\alpha}^{K_i \vee K_j} \subsetneq V_{\alpha}^H\tag{$\star$}
\end{equation}
then by Lemma \ref{tech} (6), $\forall i,j$ with $i \neq j$, $K_i \vee K_j \not \subseteq G_{(V_{\alpha}^H)}$. If $G_{(V_{\alpha}^H)} =H$ then ok, else by the boolean structure and minimality $\exists i$ such that $G_{(V_{\alpha}^H)} = K_i$. Now $L_i:=K_i^{\complement}$ (see Definition \ref{comp}) is a maximal subgroup of $G$, so by Lemma \ref{B1}, there is $\beta$ such that $G_{(V_{\beta}^{L_i})} = L_i$.

\begin{claim} \label{claim}
$\exists V_{\gamma} \le V_{\alpha} \otimes V_{\beta}$ such that $ K_i \cap G_{(V_{\gamma}^H)} ,  G_{(V_{\gamma}^H)} \cap L_i \subseteq K_i \cap L_i$. \end{claim} 
\begin{claimproof}
See the first part of \cite[Theorem 6.8]{pa} proof; it exploits \eqref{Xalpha} in a tricky way (we put this reference because we didn't find an argument which avoids the use of planar algebras).
\end{claimproof} 

\noindent By $H \subseteq G_{(V_{\gamma}^H)}$, distributivity and Claim, we conclude as follows:  
$$ G_{(V_{\gamma}^H)} =  G_{(V_{\gamma}^H)} \vee H = G_{(V_{\gamma}^H)} \vee (K_i \wedge L_i) = (G_{(V_{\gamma}^H)} \wedge K_i) \vee (G_{(V_{\gamma}^H)} \wedge L_i) $$ 
$$ \subseteq (K_i \wedge L_i) \vee (K_i \wedge L_i) = H \vee H  = H$$

Else, $\forall \alpha$, $$\sum_{i,j,i \neq j}V_{\alpha}^{K_i \vee K_j} = V_{\alpha}^H.$$ $\forall k$, $\forall (i,j)$ with $i \neq j$, $\exists s \in \{i,j\}$ with $s \neq k$, but  $V_{\alpha}^{K_i \vee K_j} \subseteq V_{\alpha}^{K_s}$, so $$ \sum_{s \neq k} V_{\alpha}^{K_s} = V^H_{\alpha}.$$
It follows that $\forall i,  \forall \alpha$, $$\sum_{j \neq i}\dim(V_{\alpha}^{K_j})  \ge \dim(V^H_{\alpha}).$$
Now if $\exists \alpha \forall i, V_{\alpha}^{K_i} \subsetneq V^H_{\alpha} $ then (by Lemma \ref{tech} (6) and minimality) $G_{(V_{\alpha}^H)} = H.$ Else $\forall \alpha \exists i, V_{\alpha}^{K_i} = V^H_{\alpha} $, but $\sum_{j \neq i}\dim(V_{\alpha}^{K_j}) \ge \dim(V^H_{\alpha})$, so $$ \sum_{j}\dim(V_{\alpha}^{K_j}) \ge 2 \dim(V^H_{\alpha}) $$ By using Lemma \ref{indexrep} and taking $V_1$ trivial, we get $$ \sum_i |G:K_i| = \sum_i [\sum_{\alpha} \dim(V_{\alpha}) \dim(V_{\alpha}^{K_i})] =   \sum_{\alpha} \dim(V_{\alpha}) [ \sum_i \dim(V_{\alpha}^{K_i})] $$  $$ \ge  n+2 \sum_{\alpha \neq 1} \dim(V_{\alpha}) \dim(V^H_{\alpha})  = 2|G:H| + (n-2).$$  
It follows that $$ \sum_{i=1}^n \frac{1}{|K_i:H|} \ge 2 + \frac{n-2}{|G:H|}$$ which contradicts the assumption because $n>2$.
\end{proof}

\begin{corollary} \label{B4}
A rank $n$ boolean interval $[H,G]$ with $|K_i:H| \ge n/2$ for any minimal overgroup $K_i$ of $H$, is linearly primitive. In particular, a boolean interval of rank $n \le 4$ is linearly primitive.
\end{corollary}
\begin{proof}
$\sum_{i} \frac{1}{\vert K_i : H  \vert} \le n \times \frac{2}{n} = 2$; the result follows by Theorem \ref{<=2}.
\end{proof}  

In the next section, we get a proof at any rank $n < 7$.

\section{The proof for small index} 
This section will prove dual Ore's theorem, for any boolean interval of rank $<7$, and then for any distributive interval of index $|G:H| < 9720$.

\begin{lemma} \label{2-2}
Let $[H,G]$ be a boolean interval of rank $2$ and let $K,L$ the atoms.
Then $(|G:K|,|G:L|)$ and $ (|K:H|,|L:H|) \neq (2,2)$. 
\end{lemma}
\begin{proof}
If $|G:K| = |G:L| = 2$, then $K$ and $L$ are normal subgroups of $G$, and so $H=K \wedge L$ is also normal. So $G/H$ is a group and $[1,G/H] = [H,G]$ as lattices, but a boolean lattice is distributive, so by Ore's theorem, $G/H$ is cyclic; but it has two  subgroups of index $2$, contradiction.   
If $|K:H| = |L:H| = 2$, then $H$ is a normal subgroup of $K$ and $L$, so of $G=H \vee K$, contradiction as above. 
\end{proof}
Note the following immediate generalization:
\begin{remark} \label{n-n} Let $[H,G]$ be boolean of rank $2$, with $K$ and $L$ the atoms. \begin{itemize}
\item If $H$ is a normal subgroup of $K$ and $L$, then $|K:H| \neq |L:H|$.
\item If $K$ and $L$ are normal subgroups of $G$ then $|G:K| \neq |G:L|$.
\end{itemize}
\end{remark}

\begin{remark} \label{profor}
Let $G$ be a finite group and $H,K$ two subgroups, then  $|H| \cdot |K| = |HK| \cdot |H \cap K|$ (Product Formula). It follows that $$|H| \cdot |K| \le |H \vee K| \cdot |H \wedge K| $$
\end{remark}

\begin{corollary} \label{decre}
Let $[H,G]$ be a boolean interval of finite groups and $A$ an atom. Any $K_1, K_2 \in [H,A^{\complement}]$ with $ K_1 \subset K_2$ satisfy $$|K_1 \vee A : K_1| \le |K_2 \vee A : K_2|$$
Moreover if $|G:A^{\complement}| = 2$ then  $|K \vee A : K| = 2$, $\forall K \in [H,A^{\complement}]$.
\end{corollary}
\begin{proof}
Suppose that $K_1 \subset K_2$. By Remark \ref{profor}, $$|K_1 \vee A | \cdot |K_2| \le |(K_1 \vee A ) \vee K_2| \cdot |(K_1 \vee A ) \wedge K_2 | $$ but $K_1 \cap K_2 = K_1$, $K_1 \cup K_2 = K_2$ and $A \wedge K_2 = H$, so by distributivity $$ |K_1 \vee A | \cdot |K_2| \le |K_2 \vee A| \cdot |K_1| $$
Finally, $A^{\complement} \vee A = G$ and $\forall K \in [H,A^{\complement}]$, $K \subset A^{\complement}$, so if $|G:A^{\complement}| = 2$, then  $$2 \le  |K \vee A : K| \le |A^{\complement} \vee A : A^{\complement}| = 2,$$ It follows that $|K \vee A : K| = 2$.
\end{proof}

\begin{lemma} \label{2<->2}
Let $[H,G]$ rank $2$ boolean with $K, L$ the atoms. Then $$ |K:H| = 2 \Leftrightarrow |G:L| = 2.$$
\end{lemma}
\begin{proof}  If $|G:L| = 2$ then $|K:H| = 2$ by Corollary \ref{decre}. \\   If  $|K:H| = 2$ then $H \triangleleft K$ and $K = H \sqcup H \tau$ with $\tau H = H \tau$ and $(H \tau)^2 = H$, so $H \tau^2 = H$ and $\tau^2 \in H$. Now $L \in (H,G)$ open, then $\tau L \tau^{-1} \in (\tau H \tau^{-1},\tau G \tau^{-1}) = (H,G)$, so by assumption $\tau L \tau^{-1} \in \{K,L\}$.  If $\tau L \tau^{-1} = K$, then $L = \tau^{-1} K \tau = K$, contradiction. So $\tau L \tau^{-1} = L$. Now $H=H \tau^2 \subset L \tau^2$, and $\tau^2 \in H \subset L$, so $L \tau^2 = L$. It follows that $\langle L, \tau \rangle = L \sqcup L \tau$. But by assumption, $G = \langle L, \tau \rangle$, so $|G:L| = 2$. \end{proof}

\begin{corollary} \label{top2} If a boolean interval $[H,G]$ admits a subinterval $[K,L]$ of index $2$, then there is an atom $A$ with $L = K \vee A$ and $|G:A^{\complement}| = 2$. 
\end{corollary}
\begin{proof}
Let $[K,L]$ be the edge of index $|L:K| = 2$. By the boolean structure, there is an atom $A \in [H,G]$ such that $L = K \vee A$. Let $$K=K_1 < K_2 < \dots < K_r = A^{\complement} $$ be a maximal chain from $K$ to $A^{\complement}$. Let $L_i = K_i \vee A$, then the interval $[K_i,L_{i+1}]$ is boolean of rank $2$, now $|L_1:K_1| = 2$, so by Lemma \ref{2<->2} $$ 2=|L_1:K_1| = |L_2:K_2| = \cdots = |L_r:K_r|=|G:A^{\complement}|.$$ 
\end{proof}

\begin{remark}
Let $[H,G]$ of index $|G:H| = 2$. Then $G = H \rtimes \mathbb{Z}/2$ if $|H|$ is odd, but it's not true in general if $|H|$ even\footnote{http://math.stackexchange.com/a/1609599/84284}. 
\end{remark}

The following theorem was pointed out by Derek Holt\footnote{http://math.stackexchange.com/a/1966655/84284}.

\begin{theorem} \label{clif}
Let $G$ be a finite group, $N$ a normal subgroup of prime index $p$ and $\pi$ an irreducible  complex represenation of $N$. Exactly one of the following occurs:
\begin{itemize}
\item[(1)] $\pi$ extends to an irreducible representation of $G$,
\item[(2)] $\mathrm{Ind}_{N}^{G}(\pi)$ is irreducible.
\end{itemize}
\end{theorem}
\begin{proof}
It is a corollary of Clifford theory, see \cite{isa} Corollary 6.19.
\end{proof}

\begin{theorem} \label{exten}
Let $[H,G]$ be a boolean interval and $L$ a coatom with $|G:L| = 2$. If $[H,L]$ is linearly primitive, then so is $[H,G]$.
\end{theorem}  
\begin{proof}
Let the atom $A := L^{\complement}$. As an immediate corollary of the proofs of Lemma \ref{2<->2} and Corollary \ref{top2}, there is $\tau \in A$ such that $\forall K \in [H,L]$, $K \tau = \tau K$ and $\tau^2 \in H \subset K$, so $K \vee A = K \sqcup K \tau$ and $G = L \sqcup L \tau$. By assumption, $[H,L]$ is linearly primitive, which means the existence of an irreducible complex representation $V$ of $L$ such that $L_{(V^H)} = H$.  

 Assume that $\pi_V$ extends to an irreducible representation $\pi_{V_+}$ of $G$. Note that $G_{(V_+^H)} = H \sqcup S \tau$ with $$S=\{l \in L \mid \pi_{V_+}(l \tau) \cdot v = v, \  \forall v \in V^H  \}$$ If $S=\emptyset$ then $G_{(V_+^H)} = H$, ok. Else $S \neq \emptyset$ and note that  $$\pi_{V_+}(l \tau) \cdot v = v \Leftrightarrow \pi_{V_+}(\tau) \cdot v = \pi_{V}(l^{-1}) \cdot v $$ but $\pi_{V_+}(\tau)(V^H) \subset V^H$ and $\tau^2 \in H$, so  $\forall l_1, l_2 \in S$ and $\forall v \in V^H$, $$\pi_{V}(l_1 l_2)^{-1} \cdot v = \pi_{V_+}(\tau^2) \cdot v = v$$ It follows that $S^2 \subset H$. Now, $HS = S$, so $HS^2 = (HS)S = S^2$, which means that $S^2$ is a disjoint union of $H$-coset, then $|H| $ divides $ |S^2|$, but $S^2 \subset H$ and $S \neq \emptyset$, so $S^2 = H$. Let $s_0 \in S$, then the maps $S \ni s \mapsto s_0s \in H$ and $H \ni h \mapsto hs_0 \in S$ are  injective, so $|S| = |H|$. If $S \neq H$, then $A = H \sqcup H \tau$ and $G_{(V_+^H)} = H \sqcup S\tau$ are two different groups containing $H$ with index $2$, contradiction with the boolean structure by Lemma \ref{2-2}. So we can assume that $H=S$. Now the extension $V_+$ is completely characterized by 
$\pi_{V_+}(\tau)$, and we can make an other irreducible extension $V_-$ characterized by $\pi_{V_-}(\tau) = -\pi_{V_+}(\tau)$. As above, $G_{(V_-^H)} = H \sqcup S' \tau$ with $$S'=\{l \in L \mid \pi_{V_-}(l \tau) \cdot v = v, \  \forall v \in V^H  \}.$$
 But $\pi_{V_-}(l \tau)= -\pi_{V_+}(l \tau)$, so $$S'=\{l \in L \mid \pi_{V_+}(l \tau) \cdot v = -v, \  \forall v \in V^H  \}.$$ Then $S \cap S' = \emptyset$, but $S=H$, so $S' \neq H$, contradiction as above.   

Next, we can assume that $\pi_V$ does not extend to an irreducible representation of $G$. So $\pi_W := \mathrm{Ind}_{L}^{G}(\pi_V)$ is irreducible by Theorem \ref{clif}. We need to check that $G_{(W^H)} = H$. We can see $W$ as $V \oplus \tau V$, with $$\pi_W(l) \cdot (v_1 + \tau v_2) = \pi_V(l) \cdot v_1 + \tau [\pi_V(\tau^{-1}l\tau) \cdot v_2], $$ with $l \in L$, and $$\pi_W(\tau) \cdot (v_1 + \tau v_2) = \pi_V(\tau^2) \cdot v_2 + \tau v_2$$ Then $$W^H = \{v_1 + \tau v_2 \in W \mid \pi_V(h) \cdot v_1 = v_1 \text{ and }  \pi_V(\tau^{-1} h \tau) \cdot v_2 = v_2, \forall h \in H \}$$ But $\tau^{-1} H \tau = H$, so $W^H = V^H \oplus \tau V^H$. Finally, according to $\pi_W(l)$ and $\pi_W(\tau)$ above, we see that $G_{(W^H)} \subset L$, and then $G_{(W^H)} = H$.
\end{proof}

\begin{remark} It seems that we can extend Theorem \ref{exten}, replacing $|G:L| = 2$ by $L \triangleleft G$ (and so $|G:L| = p$ prime), using Theorem \ref{clif} and Remark \ref{n-n}. In the proof, we should have  $K \vee A = K \sqcup K \tau \sqcup \cdots \sqcup K \tau^{p-1}$, $\tau^p \in H$, $S^p= H$ and $\pi_{V_-}(\tau) = e^{2\pi i/p}\pi_{V_+}(\tau)$. We didn't check the details because we don't need this extension. 
\end{remark}

\begin{corollary} \label{no2}
Let $[H,G]$ be a boolean interval with an atom $A$ satisfying $|A:H| = 2$.  If $[H,A^{\complement}]$ is linearly primitive, then so is $[H,G]$.\end{corollary}
\begin{proof} Immediate by Corollary \ref{top2} and Theorem \ref{exten}.
\end{proof}

One of the main result of the paper is the following:

\begin{theorem} \label{B6}
A boolean interval $[H,G]$ of rank $n<7$, is linearly primitive.
\end{theorem}
\begin{proof}
Let $K_1, \dots, K_n$ be the atoms of $[H,G]$. By Corollary \ref{no2}, we can assume that $|K_i:H| \neq 2$, $\forall i$. Now $n \le 6$ and $|K_i:H| \ge 3$, then
$$\sum_{i=1}^n \frac{1}{\vert K_i : H  \vert} \le 6 \times \frac{1}{3} = 2.$$ The result follows by Theorem \ref{<=2}.
\end{proof}

For the upper bound on the index of distributive interval we will need a former result (proved group theoretically in \cite{bp}):

\begin{theorem}{\cite[Theorem 3.24]{bp}} \label{theo2bis}
A boolean interval $[H,G]$ with a  (below) nonzero dual Euler totient is linearly primitive. $$\hat{\varphi}(H,G):=\sum_{K \in [H,G]} (-1)^{\ell(H,K)} |G:K|$$ with  $\ell(H,K)$ the rank of $[H,K]$.
\end{theorem}

\begin{conjecture} \label{lower} A rank $n$ boolean interval has $\hat{\varphi} \ge 2^{n-1} $.
\end{conjecture}

\begin{remark}
If Conjecture \ref{lower} is correct, then its lower bound is optimal, because realized by the interval $[1 \times S_2^n , S_2 \times S_3^n]$.
\end{remark}

\begin{lemma} \label{unic} Let $[H,G]$ be a boolean interval of rank $n$ and index $\prod p_i^{r_i}$ with $p_i$ prime and $\sum_i r_i = n$. Then for any atom $A$ and any $K \in [H,A^{\complement}]$, $|K \vee A:K| = p_i$ for some $i$.
\end{lemma}
\begin{proof} Let $A_1, \cdots , A_r$ be the atoms of $[H,G]$ such that $K = \bigvee_{i=1}^r A_i$, let $A_{r+1} = A$ and $A_{r+2}, \dots, A_n$ all the other atoms. By considering the corresponding maximal chain we have that $$|G:H| = |A_1:H| \cdot |A_1 \vee A_2:A_1|  \cdots |K \vee A:K|  \cdots |G:A^{\complement}_{n-1}|$$ It's a product of $n$ numbers $>1$ and the result is composed by $n$ prime numbers, so by the fundamental theorem of arithmetic, any component above is prime, then $|K \vee A:K| = p_i$ for some $i$.
\end{proof}

\begin{lemma}  \label{p^n} Let $[H,G]$ be a boolean interval of rank $n$ and index $p^n$ with $p$ prime. Then $\hat{\varphi}(H,G) = (p-1)^n>0$.
\end{lemma}
\begin{proof} 
By Lemma \ref{unic},  $\hat{\varphi}(H,G) = \sum_k (-1)^k {n \choose k}p^k = (p-1)^n$
\end{proof}
\begin{remark} Lemma \ref{p^n} is coherent with Conjecture \ref{lower} because if $p=2$ then $n=1$ by Lemma \ref{2-2}. \end{remark}

\begin{proposition}  \label{p^nq} Let $[H,G]$ be a boolean interval of rank $n$ and index $p^{n-1} q$, with $p,q$ prime and $p \le q$. Then $$\hat{\varphi}(H,G) = (p-1)^{n}  [1+\frac{q-p}{p} (1- \frac{1}{(1-p)^m})] \ge (p-1)^n > 0.$$ with $m$ be the number of coatoms $L \in [H,G]$ with $|G:L| = q$. 
\end{proposition}
\begin{proof}  If $m = 0$, then by Lemma \ref{unic}, Corollary \ref{decre} and $p \le q$, for any atom $A \in [H,G]$ and $\forall K \in [H, A^{\complement}]$, $|K \vee A : K| = p$, so $|G:H| = p^n$ and $\hat{\varphi}(H,G) = (p-1)^n$ by Lemma \ref{p^n}, ok.   

Else $m \ge 1$. We will prove the formula by induction. If $n=1$, then $m=1$ and $\hat{\varphi}(H,G) = q-1$, ok. Next, assume it is true at rank $<n$. Let $L$ be a coatom with $|G:L| = q$, then for $A = L^{\complement}$, $$\hat{\varphi}(H,G) = q\hat{\varphi}(H,L) - \hat{\varphi}(A,G)$$
Now $|L:H| = p^{n-1}$ so by Lemma \ref{p^n},  $\hat{\varphi}(H,L) = (p-1)^{n-1}$. But $|A:H| = p$ or $q$. If $|A:H| = p$ then $|G:A| = p^{n-2}q$ and by induction
$$\hat{\varphi}(A,G) = (p-1)^{n-1}  [1+\frac{q-p}{p} (1- \frac{1}{(1-p)^{m-1}})].$$ 
Else $|A:H| = q$, $|G:A| = p^{n-1}$, $m=1$ and the same formula works. Then $$\hat{\varphi}(H,G) = (p-1)^{n-1}  [q-1-\frac{q-p}{p} (1- \frac{1}{(1-p)^{m-1}})] $$ $$=  (p-1)^{n}  [\frac{q-1}{p-1}+\frac{q-p}{p} (\frac{1}{1-p}- \frac{1}{(1-p)^{m}})] $$ 
$$=  (p-1)^{n}  [\frac{q-1}{p-1} - \frac{q-p}{p}(1+\frac{1}{p-1}) +\frac{q-p}{p} (1- \frac{1}{(1-p)^m})] $$
$$=  (p-1)^{n}  [1 +\frac{q-p}{p} (1- \frac{1}{(1-p)^m})] $$ 
The result follows.
\end{proof}

\begin{definition} A chain $H_1 \subset \cdots \subset H_{r+1}$ is of type $(k_1, \dots , k_r)$ if $\exists \sigma \in S_r$ with $k_{\sigma(i)} = |H_{i+1}:H_i|$ $($so that we can choose $(k_i)_i$ increasing$)$. 
\end{definition}

\begin{remark} \label{3-4} The proof of Proposition \ref{p^nq} is working without assuming $p,q$ prime, but assuming type $(p, \dots , p, q)$ for every maximal chain of $[H,G]$. For $p$ prime and $q=p^2$ we deduce that at rank $n$ and index $p^{n+1}$, there is $1 \le m \le n$ such that  $$\hat{\varphi}(H,G) = (p-1)^{n+1}+(p-1)^{n} - (-1)^{m}(p-1)^{n+1-m} \ge (p-1)^{n+1}$$ If there is no edge of index $2$, we can also take $q=2p$ or $(p,q) = (3,4)$. 
\end{remark}

\begin{lemma} \label{check}
A boolean interval $[H,G]$ of index $|G:H| = a^nbc$ and rank $n+2$ with $3 \le a \le b \le c \le 12$, $1 \le n \le 6$ and every maximal chain of type $(a, \dots , a, b, c)$, has a dual Euler totient $\hat{\varphi}(H,G) \ge (a-1)^{n+2}.$
\end{lemma}
\begin{proof} This is checked by computer calculation using the following iterative method. Let $L$ be a coatom just that $|G:L| = c$ and $A=L^{\complement}$. Then $\hat{\varphi}(H,G) = c \hat{\varphi}(H,L)  - \hat{\varphi}(A,G)$. Now $|L:H| = a^nb$ so we can use Propoposition \ref{p^nq} formula for $\hat{\varphi}(H,L)$. Next there are three cases: $|A:H| = a,b$ or $c$. If $|A:H| = c$ then, by Corollary \ref{decre}, $\forall K \in [H,L]$, $|K \vee A:K| = c$, so $\hat{\varphi}(H,G) = (c-1)\hat{\varphi}(H,L)$. If $|A:H| = b$, then $|G:A| = a^nc$ so we can use Propoposition \ref{p^nq} formula for $\hat{\varphi}(A,G)$. Else $|A:H| = a$ and $|G:A| = a^{n-1}bc$, so we iterate the method.   \end{proof}

\begin{remark} \label{split}
Let $[H,G]$ be a boolean interval and $A$ an atom such that $\forall K \in [H,A^{\complement}]$, $|K \vee A : K| = |A:H|$. So $\hat{\varphi}(H,A^{\complement}) = \hat{\varphi}(A,G)$ and $$\hat{\varphi}(H,G) = |A:H|\hat{\varphi}(H,A^{\complement}) - \hat{\varphi}(A,G) = (|A:H|-1)\hat{\varphi}(A,G).$$
\end{remark}

\begin{corollary} \label{allsplit}
Let $[H,G]$ be a boolean interval such that for any atom $A$ and $\forall K \in [H,A^{\complement}]$, $|K \vee A : K| = |A:H|$. Then $$\hat{\varphi}(H,G) = \prod_{i=1}^n(|A_i:H|-1)>0.$$ with $A_1, \dots , A_n$ all the atoms of $[H,G]$.
\end{corollary}
\begin{proof}
By Remark \ref{split} and induction.
\end{proof}

\begin{lemma} \label{B2Lemma}
Let $[H,G]$  boolean of rank $2$ and index $<32$. Let $K,L$ be the atoms, $a=|G:K|$, $b=|G:L|$, $c=|L:H|$ and $d=|K:H|$.
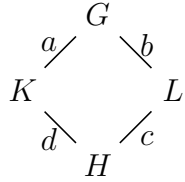
\begin{center} $\begin{tikzpicture}
\node (A1) at (0,0) {$G$};
\node (A2) at (-1,-1) {$K$};
\node (A3) at (0,-2) {$H$};
\node (A4) at (1,-1) {$L$};
\tikzstyle{segm}=[-,>=latex, semithick]
\draw [segm] (A1)to(A2); 
\draw [segm] (A2)to(A3); 
\draw [segm] (A4)to(A3); 
\draw [segm] (A1)to(A4); 
\node  at (-0.65,-1.65) {$d$};  
\node  at (-0.65,-0.4) {$a$}; 
\node  at (0.65,-0.4) {$b$}; 
\node  at (0.65,-1.65) {$c$}; 
\end{tikzpicture} 
$ \end{center}
If $a \neq 7$, then $(a,b)=(c,d)$. \\  If $a = 7$ and $a \neq c$ then $a=b=7$ and $c=d \in \{3,4 \}$.  
\end{lemma}
\begin{proof}
We can check by GAP\footnote{The GAP Group, http://www.gap-system.org, version 4.8.3, 2016.} that there are exactly $241$ boolean intervals $[H,G]$ of rank $2$ and index $|G:H|< 32$ (up to equivalence). They all satisfy $(a,b)=(c,d)$, except $[D_8,PSL_2(7)]$ and $[S_3, PSL_2(7)]$, for which $(a,b) = (7,7)$ and $(c,d) = (3,3)$ or $(4,4)$.
\end{proof}
\begin{corollary}  \label{allsplitsmall}
Let $[H,G]$ be a boolean interval having a maximal chain such that the product of the index of two different edges is $<32$, and no edge has index $7$. Then $[H,G]$ satisfies Corollary \ref{allsplit}.
\end{corollary}
\begin{proof}
Consider such a maximal chain $$H = K_0 \subset K_1 \subset \cdots \subset K_n = G$$ and $A_1, \dots, A_n$ the atoms of $[H,G]$ such that $K_{i} = K_{i-1} \vee A_i$. Now, $\forall i$ and $\forall j < i$, $[K_{j-1},K_j \vee A_i]$ is boolean of rank $2$, so by Lemma \ref{B2Lemma}, $$|K_i:K_{i-1}| = |K_{i-2} \vee A_i : K_{i-2}| = |K_{i-3} \vee A_i : K_{i-2}| = \cdots = |A_i:H|$$ Next, $\forall i$ and $\forall j \ge i$, let $L_{j-1} = K_j \wedge A_i^{\complement}$, then $[L_j,K_{j+2}]$ is boolean of rank $2$ and by Lemma \ref{B2Lemma}, $$|K_i:K_{i-1}| = |K_{i+1} : L_i| = |K_{i+2} : L_{i+1}| = \cdots = |G:A_i^{\complement}|$$ Finally, by Corollary \ref{decre}, $\forall K \in [H, A_i^{\complement}]$, $$ |A_i:H| \le |K \vee A_i : K| \le |G:A_i^{\complement}| $$ but $|A_i:H| = |K_i:K_{i-1}| =|G:A_i^{\complement}|$; the result follows. 
\end{proof}

\begin{remark} A combinatorial argument could replace the use of Corollary \ref{decre} in the proof of Corollary \ref{allsplitsmall}. 
\end{remark}

\begin{remark} \label{list}
Here is the list of all the numbers $<10125$ which are product of at least seven integers $ \ge 3$; first with exactly seven integers: 
\begin{center}
\footnotesize
\begin{tabular}{lllll}
$2187 = 3^7$ \hspace{0.1cm}       & $4860=3^5  4^1  5$ \hspace{0.1cm}  & $6480=3^4  4^2 5$  \hspace{0.1cm}  & $7776=3^5  4^1  8$ \hspace{0.1cm} & $8748= 3^6  12$  \\
$2916= 3^6  4$      & $5103=3^6  7$      & $6561=3^6  9$      & $8019=3^6  11$      & $9072=3^4  4^2  7$\\
$3645=3^6  5$       & $5184=3^4  4^3$    & $6804=3^5  4^1  7$ & $8100=3^4 4^1  5^2$ & $9216= 3^2  4^5$\\
$3888=3^5  4^2$     & $5832=3^6  8$      & $6912=3^3  4^4$    & $8505=3^5  5^1  7$  & $9477=3^6  13$\\
$4374=3^6  6$       & $6075=3^5  5^2$    & $7290=3^6  10$     & $8640=3^3  4^3  5$  & $9720 = 3^5  4^1  10$;\\
\end{tabular}
\end{center}next with exactly eight integers: \footnotesize{$6561 = 3^8,\ 8748 = 3^7 4$;} \normalsize{nothing else.}
\end{remark}

We can now prove the main theorem of the paper:

\begin{theorem} A distributive interval $[H,G]$ of index $|G:H|<9720$, is linearly primitive.
\end{theorem}
\begin{proof}
By Lemma \ref{bottomprim}, Corollary \ref{bottomBn} and Theorem \ref{B6}, we can assume the interval to be boolean of rank $n \ge 7$, and without edge of index $2$ by Corollary \ref{top2} and Theorem \ref{exten}. So by Theorem \ref{theo2bis}, it suffices to check that for every index (except $9720$) in the list of Remark \ref{list}, any boolean interval as above with this index has a nonzero dual Euler totient. We can assume the rank to be $7$, because at rank $8$, the indices $3^8$ and $3^7 4$ are checked by Lemma \ref{p^n} and Remark \ref{3-4}, and there is nothing else at rank $>8$.  Now, any maximal chain for such a boolean interval of index $3^5  4^1  5$ has type $(3,\dots, 3,4,5)$, so it is checked by Corollary \ref{allsplitsmall}. Idem for index $3^6 10$ with $(3,\dots, 3,10)$ or $(3,\dots, 3,5,6)$. The index $3^6 7$ is checked by Proposition \ref{p^nq}. For the index $3^6 12$, if there is a maximal chain of type $(3,\dots, 3,6,6)$, $6^2>32$ but using Lemma \ref{B2Lemma} with $a,b,c,d \in \{3,6\}$ we can deduce that $(a,b)=(c,d)$, so the proof of Corollary \ref{allsplitsmall} is working; else $12$ must appears in every maximal chain, so that the proof of Proposition \ref{p^nq} works with $q=12$. 
We can do the same for every index, except $3^5 4^1 7,\ 3^5 4^1 8,\ 3^5 5^1 7,\ 3^4  4^2  7,\ 3^5  4^1  10$. For index $3^5 4^1 8$, if there is a maximal chain of type $(3, \dots, 3,4,4,6)$, then ok by Corollary \ref{allsplitsmall}, else (because there is no edge of index $2$) every maximal chain is of type $(3, \dots, 3,4,8)$, so ok by Lemma \ref{check}. We can do the same for every remaining index except $3^5  4^1  10 = 9720$, the expected upper bound.
\end{proof}

\begin{remark}
The tools above don't check $3^5  4^1  10$ because the possible maximal chain types are $(3, \dots, 3,4,5,6)$, $(3, \dots, 3,4,10)$ and $(3, \dots, 3,5,8)$. The first is ok by Corollary \ref{allsplitsmall}, but not the two last because $4 \cdot 10 = 5 \cdot 8 = 40 >32$. So there is not necessarily a unique maximal chain type, and Lemma \ref{check} can't be applied. Nevertheless, more intensive computer investigation can probably leads beyond $9720$. \end{remark}


%
%
%

\section{Acknowledgments} 
I would like to thank Derek Holt for showing me a theorem on representation theory used in this paper. This work is supported by the Institute of Mathematical Sciences, Chennai. 



\begin{bibdiv}
\begin{biblist}


\bib{bp}{article}{
   author={Balodi, Mamta},
   author={Palcoux, Sebastien},
   title={On boolean intervals of finite groups},
   pages={25pp},
   journal={arXiv:1604.06765v5}, 
   note={submitted to Trans. Amer. Math. Soc.},
}
\bib{isa}{book}{
   author={Isaacs, I. Martin},
   title={Character theory of finite groups},
   note={Corrected reprint of the 1976 original [Academic Press, New York;
   MR0460423 (57 \#417)]},
   publisher={Dover Publications, Inc., New York},
   date={1994},
   pages={xii+303},
   isbn={0-486-68014-2},
   review={\MR{1280461}},
}
\bib{or}{article}{
   author={Ore, Oystein},
   title={Structures and group theory. II},
   journal={Duke Math. J.},
   volume={4},
   date={1938},
   number={2},
   pages={247--269},
   issn={0012-7094},
   review={\MR{1546048}},
   doi={10.1215/S0012-7094-38-00419-3},
}
\bib{pa}{article}{
   author={Palcoux, Sebastien},
   title={Ore's theorem for cyclic subfactor planar algebras and applications},
   pages={50pp},
   journal={arXiv:1505.06649v10}, 
   note={submitted to Pacific J. Math.},
}
\bib{sta}{book}{
   author={Stanley, Richard P.},
   title={Enumerative combinatorics. Volume 1},
   series={Cambridge Studies in Advanced Mathematics},
   volume={49},
   edition={2},
   publisher={Cambridge University Press, Cambridge},
   date={2012},
   pages={xiv+626},
   isbn={978-1-107-60262-5},
   review={\MR{2868112}},
}
\end{biblist}
\end{bibdiv}
\end{document}